\documentclass{amsart}
\usepackage{latexsym,amsmath,amssymb}
\newtheorem{thm}{Theorem}[section]
\newtheorem{cor}[thm]{Corollary}

\newtheorem{lem}[thm]{Lemma}
\newtheorem{prop}[thm]{Proposition}
\theoremstyle{definition}
\theoremstyle{remark}

\numberwithin{equation}{section}

\begin{document}

\title[Essential norm of  generalized Integral type Operators ]
{Essential norm of generalized Integral type  Operators from Mixed-Norm
to Zygmund-Type Spaces}
\author{\sc\bf Y. Estaremi, M. Hassanlou and M. S. Al Ghafri}
\address{\sc Y. Estaremi}
\email{y.estaremi@gu.ac.ir}
\address{Department of Mathematics, Faculty of Sciences, Golestan University, Gorgan, Iran.}
\address{ \sc M. Hassanlou}
\email{m.hassanlou@urmia.ac.ir}
\address{Engineering Faculty of Khoy, Urmia University of Technology, Urmia, Iran}
\address{ \sc M. S. Al Ghafri}
\email{mo86said@gmail.com}
\address{Department of Mathematics, Sultan Qaboos University, Muscat, Oman}

\thanks{}

\thanks{}



\date{}

\dedicatory{}

\commby{}

\begin{abstract}
In this paper, we consider the generalized integration operator from mixed-norm space into Zygmund-type and Bloch-type spaces and find an estimation for the essential norm
of this operator.
\end{abstract}

\maketitle

\section{ \sc\bf Introduction and Preliminaries}
We denote by $\mathbb{D}$ the open unit disc in the complex plan $\mathbb{C}$ and $H(\mathbb{D})$ the space of analytic functions on $\mathbb{D}$. Every analytic function (self-map) $\varphi : \mathbb{D} \rightarrow \mathbb{D}$ induces a linear operator via composition called composition operator $C_{\varphi}$, $C_{\varphi} f = f \circ \varphi$. A complete history of these operators on the spaces of analytic functions can be found in \cite{cowen}. Throughout the recent decades there has been a significant interest on investigating the operator properties of composition operator and generalization of it while they act on some spaces of analytic or measurable functions. The above operator can be generalized  by
$$\left(C_\varphi^g f\right)(z)=\int_{0}^{z}f'(\varphi(\xi))g(\xi)\ d\xi,\quad f\in H(\mathbb{D}), z\in\mathbb{D},$$
where $g\in H(\mathbb{D})$, see \cite{LS10}. Also, if we replace n-th derivative with 1-th derivative, then we have
$$\left(C_{\varphi,g}^n f\right)(z)=\int_{0}^{z}f^{(n)}(\varphi(\xi))g(\xi)\ d\xi,\quad f\in H(\mathbb{D}), z\in\mathbb{D},$$
where $n \in \mathbb{N}_0$. If $n=0$, then $C_{\varphi,g}^0$ is the Volterra composition operator. Also for $\varphi(z) =z$ and $g'$ we have Volterra type operator $T_g$,
$$T_g f(z) = \int_0^z f(w) g'(w) dw. $$
  $C_{\varphi,g}^n$ is called generalized integration  operator, see for example \cite{GL, zhu10} and references therein. A natural question dealing with this kind of operators or other classes of operators is that, under which conditions they are bounded, compact, have closed range and $\cdots$. Another aspect is computing or estimating norm, essential norm, numerical range, $\cdots$ of the operator. All characterizations are in terms of symbols induce the operators.

The aim of this paper  is to estimate the essential norm of generalized integration operators from mixed-norm spaces, which include some well-known spaces of analytic functions, into weighted Zygmund space. As a result, compactness criteria for the operator, has been proved in \cite{GL}, is obtained.

A positive continuous function $\phi$ on $[0,1)$ is called normal if there are two constants $b > a > 0$ such that
\begin{itemize}
  \item [(i)] $\frac{\phi(r)}{(1-r)^a}$ is non-increasing in $[0,1)$ and
  $\frac{\phi(r)}{(1-r)^a} \downarrow 0$
  \item [(ii)] $\frac{\phi(r)}{(1-r)^b}$ is non-decreasing  in $[0,1)$ and
  $\frac{\phi(r)}{(1-r)^b} \uparrow \infty$
\end{itemize}
as $r \rightarrow 1^-$. For $0 < p,q< \infty$ and $\phi$  normal, the mixed-norm space $H(p,q,\phi)$ consists of all analytic functions in $\mathbb{D}$ for which
$$ \| f \|_{p,q, \phi}^p = \int_0^1 M_q^p (f,r) \frac{\phi^p (r)}{1-r} dr < \infty $$
where
$$ M_q (f,r) =\left ( \frac{1}{2\pi} \int_0^{2\pi} |f(r e^{i\theta})|^q d \theta \right)^{1/q}. $$
If $1\leq p < \infty$, then  $H(p,q,\phi)$ becomes a Banach space equipped with the norm
$\| .\|_{p,q,\phi}$. In the case $0 < p<1$, $\| .\|_{p,q,\phi}$ is a quasi-norm on
$H(p,q,\phi)$ and this space is a Fr$\acute{e}$chet space but is not a Banach space.
For $\alpha>-1$, $\phi(r) = (1-r)^{(\alpha+1)/p}$ is a normal function. If $p=q$, then $H(p,p,(1-r)^{(\alpha+1)/p})$ is weighted Bergman space $A_{\alpha}^p$ which is defined by
$$ A_{\alpha}^p = \{ f \in H(\mathbb{D}): \| f \|_{A_{\alpha}^p}^p = \int_{\mathbb{D}} |f(z)|^p (1-|z|^2)^{\alpha} dA(z) < \infty \},$$
where $dA$ is the normalized area measure on $\mathbb{D}$.

A positive continuous function $\mu$ on $\mathbb{D}$ is called weight.
For a  weight $\mu$, the weighted Zygmund space $\mathcal{Z}_{\mu}$, Zygmund-type space,  is the space of $f \in H(\mathbb{D})$ such that
$$ \sup_{z \in \mathbb{D}} \mu(|z|) |f''(z)|<\infty. $$
The weighted Zygmond space $\mathcal{Z}_{\mu}$ is a Banach space with respect to the norm:
$$ \| f \|_{\mathcal{Z}_{\mu}} = |f(0)| + |f'(0)|+ \sup_{z \in \mathbb{D}} \mu(|z|) |f''(z)|.$$
The weighted Bloch space $\mathcal{B}_{\mu}$ defined by
$$ \mathcal{B}_{\mu} = \{ f \in H(\mathbb{D}): \| f \|_{\mathcal{Z}_{\mu}} = |f(0)| + \sup_{z \in \mathbb{D}} \mu(|z|) |f'(z)| < \infty \}.  $$
Note that if $\mu(z) = (1-z^2)$, then we obtain classical Zygmund space $\mathcal{Z}$ and Bloch space $\mathcal{B}$.
For a complete information about classical Bloch space, see \cite{zhu1}.

Hu in \cite{hu} has characterized those holomorphic symbols on the unit ball for which the induced extended ces$\grave{a}$ro operator  is bounded (compact) on mixed norm space. The boundedness of  ces$\grave{a}$ro operators on mixed norm space is studied in \cite{shi}. Boundedness and compactness of the generalized
integration operator from a mixed norm space into a weighted Zygmund space investigated by Gu and Liu in \cite{GL}. More results on integral type operators acting on mixed norm,  Zygmond and Bloch spaces can be found in \cite{bl,LS10, liu10, liu11, st11, st10, zhu10}.

Gu and Liu in \cite{GL} have characterized compact generalized integration operators $C_{\varphi,g}^n$ between mixed norm and weighted Zygmund spaces. Indeed, they have proved that $C_{\varphi,g}^n :H(p,q,\phi) \rightarrow \mathcal{Z}_{\mu}$ is compact if and only if
\begin{equation*}
  A = \lim_{|\varphi(z)| \rightarrow 1} \frac{\mu(|z|) |\varphi'(z) g(z)|}{\phi(|\varphi(z)|) (1-|\varphi(z)|^2)^{1/q + n+1}} =0
\end{equation*}
\begin{equation*}
  B = \lim_{|\varphi(z)| \rightarrow 1} \frac{\mu(|z|) | g'(z)|}{\phi(|\varphi(z)|) (1-|\varphi(z)|^2)^{1/q + n}} =0
\end{equation*}
In this paper, motivated by the above results, we find an approximation of the essential norm of the generalized integration operators $C_{\varphi,g}^n :H(p,q,\phi) \rightarrow \mathcal{Z}_{\mu}$ and prove that
$$ \| C_{\varphi,g}^n \|_e \approx \max \{ A, B \},$$
it means that there exists constants $C,D$ such that $CB\leq A\leq DB$ or equivalently $CA\leq B\leq DA$.

Recall that for any operator $T$ between two Banach spaces $X$ and $Y$, the essential norm of $T$ is denoted by $\|T\|_{e,X\to Y}$ and is defined as follows
$$\|T\|_{e,X \rightarrow Y}=\inf\{\|T-S\|:\ S\ \text{is a compact operator from} \ X \ \text{to}\ Y\}.$$
The operator $T$ is compact if and only if $\|T\|_{e,X\to Y} =0$.

\section{ \sc\bf Essential norm of $C_{\varphi,g}^n :H(p,q,\phi) \rightarrow \mathcal{Z}_{\mu}$ or $\mathcal{B}_{\mu}$  }

 In this section we give an estimation for the essential norm of the operator $C_{\varphi,g}^n :H(p,q,\phi) \rightarrow \mathcal{Z}_{\mu}$ or $\mathcal{B}_{\mu}$. First we recall a fundamental lemma for later use in the paper.


\begin{lem} \label{l20}\cite{st11}
Assume that $p,q \in (0,\infty)$, $\phi$ is normal and $f \in H(p,q,\phi)$. Then
for each $n \in \mathbb{N}_0$, there is a positive constant $C$ in dependent of $f$
such that
\begin{equation*}
  |f^{(n)}(z)| \leq C \frac{\| f \|_{p,q,\phi}}{\phi(|z|) (1-|z|^2)^{1/q +n}}, \ \ \ z \in \mathbb{D}.
\end{equation*}
\end{lem}
The following lemma can be obtained by a standard argument, see \cite{cho}.
\begin{lem}  \label{l10}
Let $\varphi$ be a analytic self-map of $\mathbb{D}$. Then
$C_{\varphi,g}^n :H(p,q,\phi) \rightarrow \mathcal{Z}_{\mu}$ is compact if and only if $C_{\varphi,g}^n :H(p,q,\phi) \rightarrow \mathcal{Z}_{\mu}$ is bounded and
for any bounded sequence $\{ f_k \}$ in $H(p,q,\phi)$ which converges uniformly to
zero on compact subsets of $\mathbb{D}$ as $k \rightarrow \infty$,  we have
$\| C_{\varphi,g}^n f_k \|_{\mathcal{Z}_{\mu}} \rightarrow 0 $ as $k \rightarrow \infty$.
\end{lem}
By the above observations, in the next Proposition, we get that every bounded generalized integration operator $C_{\varphi,g}^n$ with $\|\varphi\|_{\infty}<1$, is compact.

\begin{prop}\label{p10}
Let  $p,q \in (0,\infty)$, $\phi$ is normal, $g\in H(\mathbb{D})$ and $\varphi$ be a analytic self-map of $\mathbb{D}$.  If $\|\varphi\|_{\infty}<1$ and the operator $C_{\varphi,g}^n :H(p,q,\phi) \rightarrow \mathcal{Z}_{\mu}$ is bounded then
it is compact
\end{prop}
\begin{proof}
Since $C_{\varphi,g}^n$ is bounded, then there exists a positive constant $C$ such that for
\begin{equation}\label{r11}
\|  C_{\varphi,g}^n f \|_{\mathcal{Z}_{\mu}} \leq  C \|f \|_{p,q,\phi},
\end{equation}
for all $f \in H(p,q,\phi)$. Let $f_1(z)=\frac{z^n}{n!} $. Then we have
\begin{align}\label{r14}
 \sup_{z \in \mathbb{D}} \mu(|z|)|g'(z)| = \| C_{\varphi,g}^n f_1 \|_{\mathcal{Z}_{\mu}} \leq C \|f_1 \|_{p,q,\phi} < \infty.
\end{align}
Similarly for  $f_2(z)=\frac{z^{n+1}}{(n+1)!}$, by using triangle inequality and the fact that $|\varphi(z)|<1$, we get that
\begin{align}\label{r15}
\sup_{z \in \mathbb{D}} \mu (|z|) |\varphi'(z) g(z)|<\infty.
\end{align}
Now assume that ${\{f_k\}}_{k\in {\mathbb N}}$ is a bounded sequence in $H(p,q,\phi)$ such that converges to $0$, uniformly on compact subsets of $\mathbb{D}$, as $k\to\infty$. This implies that  $\{f_k'\}_{k=1}^\infty$ and $\{f_k''\}_{k=1}^\infty$ converge uniformly to $0$, on compact subsets of $\mathbb{D}$, as $k\to\infty$.
By the Lemma \eqref{r14} and Proposition \eqref{r15} we obtain
\begin{align*}
\nonumber
\| C_{\varphi,g}^n f_{k} \|_{\mathcal{Z}_\mu}=&|( C_{\varphi,g}^n f_{k})' (0)|+\sup_{z\in \mathbb{D}} \mu(z) |( C_{\varphi,g}^n f_{k})''(z)|\\
\leq & | f_k^{(n)} (\varphi(0)) g(0)| + \sup_{z\in \mathbb{D}} \mu(|z|) |(C_{\varphi,g}^n f_k )''(z)| \\
\leq  &  | f_k^{(n)} (\varphi(0)) g(0)| + \sup_{z\in \mathbb{D}} \mu(|z|) |f_k^{(n+1)} (\varphi(z)) \varphi'(z) g(z)|\\
&  +  \sup_{z\in \mathbb{D}} \mu(|z|) |f_k^{(n)} (\varphi(z)) g'(z)| \\
= & | f_k^{(n)} (\varphi(0)) g(0)| + \sup_{|\varphi(z)| \leq \rho} \mu(|z|) |f_k^{(n+1)} (\varphi(z)) \varphi'(z) g(z)|\\
&  +  \sup_{|\varphi(z)| \leq \rho} \mu(|z|) |f_k^{(n)} (\varphi(z)) g'(z)|,
\end{align*}
which tends to $0$, as $k \rightarrow \infty$. Here $\rho=\|\varphi\|_\infty$,  $\rho\in(0,1)$. Hence by Lemma \ref{l10}, $C_{\varphi,g}^n$ is compact.
\end{proof}
Now, in the next Theorem we estimate the essential norm of $C_{\varphi,g}^n :H(p,q,\phi) \rightarrow \mathcal{Z}_{\mu}$. We use $\|C_{\varphi,g}^n\|_{e}$ for the essential norm
of $C_{\varphi,g}^n :H(p,q,\phi) \rightarrow \mathcal{Z}_{\mu}$ instead of $\|C_{\varphi,g}^n\|_{e,H(p,q,\phi)\to \mathcal{Z}_{\mu}}$.

\begin{thm} \label{th1}
Let  $p,q \in (0,\infty)$, $\phi$ is normal and $\varphi$ be a analytic self-map of $\mathbb{D}$.  If  the operator $C_{\varphi,g}^n :H(p,q,\phi) \rightarrow \mathcal{Z}_{\mu}$ is bounded then
\begin{align*}
  \| C_{\varphi,g}^n \|_{e} \approx \max  \{ &  \limsup_{|\varphi(z)| \rightarrow 1} \frac{\mu(|z|) |\varphi'(z) g(z)|}{\phi(|\varphi(z)|) (1-|\varphi(z)|^2)^{1/q + n+1}}, \\ &  \limsup_{|\varphi(z)| \rightarrow 1} \frac{\mu(|z|) | g'(z)|}{\phi(|\varphi(z)|) (1-|\varphi(z)|^2)^{1/q + n}} \}.
\end{align*}
\end{thm}
\begin{proof}
If $\| \varphi \|_{\infty} < 1$, then by the Proposition \ref{p10}, the operator $C_{\varphi,g}^n$ is compact and therefore $\| C_{\varphi,g}^n \|_{e} =0$. Also by Theorem 2.7 of \cite{GL}, we have
$$\limsup_{|\varphi(z)| \rightarrow 1} \frac{\mu(|z|) |\varphi'(z) g(z)|}{\phi(|\varphi(z)|) (1-|\varphi(z)|^2)^{1/q + n+1}}=0$$
and
$$\limsup_{|\varphi(z)| \rightarrow 1} \frac{\mu(|z|) | g'(z)|}{\phi(|\varphi(z)|) (1-|\varphi(z)|^2)^{1/q + n}}=0.$$
Now assume that $\| \varphi \|_{\infty}=1$.
Let $\{z_j\}_{j\in\mathbb{N}}$ be a sequence in $\mathbb{D}$ such that $|\varphi(z_j)|\rightarrow 1$, as $j\rightarrow \infty$.
Since $\phi$ is a normal function, then we have a positive constant $b$ from the definition. Put
$\alpha = 1/q + b+1$ and define the sequence $\{ f_k \}$ as follows
\begin{align*}
  f_k (z) =   \frac{(1-|\varphi(z_k)|^2)^{b+1}}{\phi(|\varphi(z_k)|)}
 \left ( \frac{1}{(1-\overline{\varphi(z_k)}z)^{\alpha}} -
 \frac{\alpha (1-|\varphi(z_k)|^2)}{(\alpha +n)(1-\overline{\varphi(z_k)}z)^{\alpha+1}} \right).
\end{align*}
 Direct calculations shows that
\begin{align*}
  f_k^{(n)} (z) = &  \frac{(1-|\varphi(z_k)|^2)^{b+1}}{\phi(|\varphi(z_k)|)} \frac{\alpha(\alpha+1) \cdots (\alpha +n-1)\overline{\varphi(z_k)}^n}{(1-\overline{\varphi(z_k)}z)^{\alpha+n}} \\
 & - \frac{(1-|\varphi(z_k)|^2)^{b+1}}{\phi(|\varphi(z_k)|)} \frac{\alpha(\alpha+1) \cdots (\alpha+n-1)(\alpha+n) (1-|\varphi(z_n)|^2)\overline{\varphi(z_n)}^n}{(\alpha+n)
  (1-\overline{\varphi(z_k)}z)^{\alpha+n+1}} \\
 f_k^{(n+1)} (z) = &  \frac{(1-|\varphi(z_k)|^2)^{b+1}}{\phi(|\varphi(z_k)|)} \frac{\alpha(\alpha+1) \cdots (\alpha +n)\overline{\varphi(z_k)}^{n+1 }  }{(1-\overline{\varphi(z_k)}z)^{\alpha+n}} \\
 & - \frac{(1-|\varphi(z_k)|^2)^{b+1}}{\phi(|\varphi(z_k)|)} \frac{\alpha(\alpha+1) \cdots (\alpha+n-1)(\alpha+n) (1-|\varphi(z_k)|^2)\overline{\varphi(z_k)}^n}{(\alpha+n)
  (1-\overline{\varphi(z_k)}z)^{\alpha+n+1}}.
\end{align*}
It is not difficult to show that $\{ f_k \}$ is a bounded sequence in $H(p,q,\phi)$, i.e., there exists a positive constant $C$ such that
\begin{equation}\label{r22}
  \sup_{k} {\| f_k \|_{p,q,\phi}} \leq C.
\end{equation}
Also, it is convergent to zero uniformly on compact subsets of $\mathbb{D}$. Moreover, we have
\begin{equation*}
  f_k^{(n)} (\varphi(z_k)) = 0, \ \  f_k^{(n+1)} (\varphi(z_k)) = -
  \frac{\alpha (\alpha +1)(\alpha +2) \cdots (\alpha + n-1) \overline{\varphi(z_k)}^{n+1}}{\phi(|\varphi(z_k)|) (1-|\varphi(z_k)|^2)^{1/q + n+1}}.
\end{equation*}
Therefore for any compact operator $T: H(p,q,\phi) \rightarrow \mathcal{Z}_{\mu} $ we have
$$ \lim_{n \rightarrow \infty} \| T f_k \|_{\mathcal{Z}_{\mu}} =0.$$
By these observations we get that
\begin{align*}
  \|  C_{\varphi,g}^n-T\| &\succeq
  \limsup_{k \rightarrow \infty} \|  C_{\varphi,g}^n f_k - T f_k \|_{\mathcal{Z}_{\mu}} \\
  &\succeq
 \limsup_{k \rightarrow \infty}\|  C_{\varphi,g}^n f_k \|_{\mathcal{Z}_{\mu}}-
 \limsup_{k \rightarrow \infty}\| T f_k \|_{\mathcal{Z}_{\mu}}\nonumber\\
 &\succeq  \limsup_{k \rightarrow \infty} \frac{\mu(|z_k|) |\varphi'(z_k) g(z_k)| |\varphi(z_k)|^{n+1}}{ \phi(|\varphi(z_k)|) (1-|\varphi(z_k)|^2)^{1/q + n+1} }.
 \end{align*}
And so
\begin{align} \label{r11}
\| C_{\varphi,g}^n \|_{e}= \inf_{T}\| C_{\varphi,g}^n-T\|  \succeq
  \limsup_{|\varphi(z)| \rightarrow 1} \frac{\mu(|z|) |\varphi'(z) g(z)| }{ \phi(|\varphi(z)|) (1-|\varphi(z)|^2)^{1/q + n+1}}.
 \end{align}
Now we define the sequence $\{ h_k \}$ as follows
\begin{align*}
  h_k (z) =   \frac{(1-|\varphi(z_k)|^2)^{b+1}}{\phi(|\varphi(z_k)|)}
 \left ( \frac{\alpha +n +1 }{(1-\overline{\varphi(z_k)}z)^{\alpha}}
 -
 \frac{\alpha (1-|\varphi(z_k)|^2)}{(1-\overline{\varphi(z_k)}z)^{\alpha+1}} \right).
\end{align*}
It is easy to see that
\begin{align*}
  h_k^{(n)} (z) = &  \frac{(1-|\varphi(z_k)|^2)^{b+1}}{\phi(|\varphi(z_k)|)} \frac{\alpha(\alpha+1) \cdots (\alpha +n-1)(\alpha + n+1)\overline{\varphi(z_k)}^n}{(1-\overline{\varphi(z_k)}z)^{\alpha+n}} \\
 & - \frac{(1-|\varphi(z_k)|^2)^{b+1}}{\phi(|\varphi(z_k)|)} \frac{\alpha(\alpha+1) \cdots (\alpha+n) (1-|\varphi(z_n)|^2)\overline{\varphi(z_n)}^n}{(\alpha+n)
  (1-\overline{\varphi(z_k)}z)^{\alpha+n+1}} \\
 h_k^{(n+1)} (z) = &  \frac{(1-|\varphi(z_k)|^2)^{b+1}}{\phi(|\varphi(z_k)|)} \frac{\alpha(\alpha+1) \cdots (\alpha +n+1)\overline{\varphi(z_k)}^{n+1 }   }{(1-\overline{\varphi(z_k)}z)^{\alpha+n+1}} \\
 & - \frac{(1-|\varphi(z_k)|^2)^{b+1}}{\phi(|\varphi(z_k)|)} \frac{\alpha(\alpha+1) \cdots (\alpha+n+1)) (1-|\varphi(z_k)|^2)\overline{\varphi(z_k)}^{n+1}}{
  (1-\overline{\varphi(z_k)}z)^{\alpha+n+2}}.
\end{align*}
Then $\{ h_k \}$ is a bounded sequence in $H(p,q,\phi)$  converging to zero uniformly on compact subsets of $\mathbb{D}$  and
\begin{equation*}
  h_k^{(n)} (\varphi(z_k)) = \frac{\alpha (\alpha +1)(\alpha +2) \cdots (\alpha + n-1) \overline{\varphi(z_k)}^{n}}{\phi(|\varphi(z_k)|) (1-|\varphi(z_k)|^2)^{1/q + n}}, \ \  h_k^{(n+1)} (\varphi(z_k)) = 0.
\end{equation*}
Hence
$$ \lim_{n \rightarrow \infty} \| T h_k \|_{\mathcal{Z}_{\mu}} =0, $$
for any compact operator $T: H(p,q,\phi) \rightarrow \mathcal{Z}_{\mu} $.
By definition of the operator norm we have
\begin{align*}
  \|  C_{\varphi,g}^n-T\| &\succeq
  \limsup_{k \rightarrow \infty} \|  C_{\varphi,g}^n h_k - T h_k \|_{\mathcal{Z}_{\mu}} \\
  &\succeq
 \limsup_{k \rightarrow \infty}\|  C_{\varphi,g}^n h_k \|_{\mathcal{Z}_{\mu}}-
 \limsup_{k \rightarrow \infty}\| T h_k \|_{\mathcal{Z}_{\mu}}\nonumber\\
 &\succeq  \limsup_{k \rightarrow \infty} \frac{\mu(|z_k|) | g(z_k)| |\varphi(z_k)|^{n}}{ \phi(|\varphi(z_k)|) (1-|\varphi(z_k)|^2)^{1/q + n+} }.
 \end{align*}
Therefore, using the definition of the essential norm we obtain
\begin{align} \label{r12}
\| C_{\varphi,g}^n \|_{e}= \inf_{T}\| C_{\varphi,g}^n-T\|  \succeq
  \limsup_{|\varphi(z)| \rightarrow 1} \frac{\mu(|z|) | g(z)| }{ \phi(|\varphi(z)|) (1-|\varphi(z)|^2)^{1/q + n} }.
 \end{align}
From \eqref{r11} and \eqref{r12}, the lower estimate of the essential norm is achieved.

Suppose that $\{ r_j \}$ be a sequence in $(0,1)$ and $r_j \rightarrow1$ as $j \rightarrow \infty$. Since $C_{\varphi,g}^n :H(p,q,\phi)\rightarrow {{\mathcal Z}}_{\mu }$ is bounded, then it is clear that $C_{r_j \varphi,g}^n :H(p,q,\phi) \rightarrow {{\mathcal Z}}_{\mu }$ is also bounded. Also, since $\| r_j \varphi \|_{\infty} < 1$, then by Proposition \ref{p10}, we get that $C_{\varphi,g}^n :H(p,q,\phi) \rightarrow {{\mathcal Z}}_{\mu }$ is a compact operator for any $j$. Thus
\begin{align*}
  \| C_{\varphi,g}^n  \|_e \leq \| C_{\varphi,g}^n  - C_{r_j\varphi,g}^n  \| = \sup_{\| f\|_{p,q,\phi}\leq 1} \| (C_{\varphi,g}^n - C_{r_j\varphi,g}^n)f \|_{\mathcal{Z}_{\mu}} , \ \ \ j \in \mathbb{N}.
\end{align*}
Now let $f \in H(p,q,\phi)$ and $\| f \|_{Q_K(p,q)} \leq 1$. Then
\begin{align} \label{r153}
\nonumber
 \|( C_{\varphi,g}^n - C_{r_j \varphi,g}^n)f \|_{\mathcal{Z}_\mu}
= &  | (( C_{\varphi,g}^n  -  C_{r_j \varphi,g}^n )f)(0) | + |(( C_{\varphi,g}^n -  C_{r_j\varphi,g}^n)f)'(0)| \nonumber
 \\
 & +\sup_{z\in \mathbb{D}}\mu(|z|) |{(( C_{\varphi,g}^n - C_{r_j\varphi,g}^n )f)''}(z)|\nonumber
 \\
\leq &  |(f^{(n)}(\varphi(0))  - f^{(n)}(r_j \varphi(0))) g(0)| \nonumber \\
& +\sup_{z\in \mathbb{D}}\mu(|z|) |\varphi'(z) g(z)(f^{(n+1)} (\varphi(z)) - r_j f^{(n+1)} (r_j\varphi(z))) |  \nonumber \\
& + \sup_{z\in \mathbb{D}}\mu(|z|) | g'(z)(f^{(n)} (\varphi(z)) -  f^{(n)} (r_j\varphi(z))) | \nonumber \\
= &  |(f^{(n)}(\varphi(0))  - f^{(n)}(r_j \varphi(0))) g(0)| \nonumber \\
& +\sup_{ |\varphi(z)| \leq \delta}\mu(|z|) |\varphi'(z) g(z)(f^{(n+1)} (\varphi(z)) - r_j f^{(n+1)} (r_j\varphi(z))) |  \nonumber \\
& + \sup_{\delta < |\varphi(z)| < 1}\mu(|z|) |\varphi'(z) g(z)(f^{(n+1)} (\varphi(z)) - r_j f^{(n+1)} (r_j\varphi(z))) |  \nonumber \\
& + \sup_{|\varphi(z)| \leq \delta}\mu(|z|) | g'(z)(f^{(n)} (\varphi(z)) -  f^{(n)} (r_j\varphi(z))) | \nonumber \\
& + \sup_{\delta < |\varphi(z)| < 1}\mu(|z|) | g'(z)(f^{(n)} (\varphi(z)) -  f^{(n)} (r_j\varphi(z))) |,
\end{align}
where $\delta \in (0,1)$, is fixed.
Since $f_{r_j} \rightarrow f$ uniformly on compact
subsets of $\mathbb{D}$, where $f_r (z) = f(rz)$, then
$$ |(f^{(n)}(\varphi(0))  - f^{(n)}(r_j \varphi(0))) g(0)|\rightarrow 0$$
as $j \rightarrow \infty$. Also
$$ \limsup_{ j \rightarrow \infty} \sup_{ |\varphi(z)| \leq \delta}\mu(|z|) |\varphi'(z) g(z)(f^{(n+1)} (\varphi(z)) - r_j f^{(n+1)} (r_j\varphi(z))) | =0 $$
$$ \limsup_{ j \rightarrow \infty} \sup_{|\varphi(z)| \leq \delta}\mu(|z|) | g'(z)(f^{(n)} (\varphi(z)) -  f^{(n)} (r_j\varphi(z))) | =0.$$

Here we use \eqref{r14} and \eqref{r15} and the fact that  $f_{r_j}^{(n+1)} \rightarrow f^{(n+1)}$ and $f_{r_j}^{(n)} \rightarrow f^{(n)}$ uniformly on compact
subsets of $\mathbb{D}$. Also by \ref{r153} and Lemma \ref{l20} we have
\begin{align} \label{r153}
\nonumber
 \|( C_{\varphi,g}^n - C_{r_j \varphi,g}^n)f \|_{\mathcal{Z}_\mu}
\leq &  \sup_{\delta < |\varphi(z)| < 1}\mu(|z|) |\varphi'(z) g(z)(f^{(n+1)} (\varphi(z)) - r_j f^{(n+1)} (r_j\varphi(z))) |  \nonumber \\
& + \sup_{\delta < |\varphi(z)| < 1}\mu(|z|) | g'(z)(f^{(n)} (\varphi(z)) -  f^{(n)} (r_j\varphi(z))) | \nonumber \\
\leq &  \sup_{\delta < |\varphi(z)| < 1}\mu(|z|) |\varphi'(z) g(z)f^{(n+1)} (\varphi(z)) |  \nonumber \\
& +  \sup_{\delta < |\varphi(z)| < 1}\mu(|z|) |\varphi'(z) g(z)r_j f^{(n+1)} (r_j\varphi(z)) |  \nonumber \\
& + \sup_{\delta < |\varphi(z)| < 1}\mu(|z|) | g'(z)f^{(n)} (\varphi(z))) | \nonumber \\
& + \sup_{\delta < |\varphi(z)| < 1}\mu(|z|) | g'(z)   f^{(n)} (r_j\varphi(z)) | \nonumber \\
\leq & 2C \sup_{\delta < |\varphi(z)| < 1}  \frac{\mu(|z|) |\varphi'(z) g(z)|}{\phi(|\varphi(z)|) (1-|\varphi(z)|^2)^{1/q + n+1}} \nonumber \\
& +   2C
\sup_{\delta < |\varphi(z)| < 1} \frac{\mu(|z|) | g'(z)|}{\phi(|\varphi(z)|) (1-|\varphi(z)|^2)^{1/q + n}}.
\end{align}
So
\begin{align*}
  \| C_{\varphi,g}^n  \|_e \leq &
 \sup_{\| f\|_{p,q,\phi}\leq 1} \|( C_{\varphi,g}^n - C_{r_j \varphi,g}^n)f \|_{\mathcal{Z}_\mu} \\
  \leq & 2C \sup_{\delta < |\varphi(z)| < 1}  \frac{\mu(|z|) |\varphi'(z) g(z)|}{\phi(|\varphi(z)|) (1-|\varphi(z)|^2)^{1/q + n+1}} \nonumber \\
& +   2C
\sup_{\delta < |\varphi(z)| < 1} \frac{\mu(|z|) | g'(z)|}{\phi(|\varphi(z)|) (1-|\varphi(z)|^2)^{1/q + n}}.
\end{align*}
Therefore as $\delta \rightarrow 1$ we obtain
\begin{align*}
  \| C_{\varphi,g}^n \|_{e} \preceq \max  \{ &  \limsup_{|\varphi(z)| \rightarrow 1} \frac{\mu(|z|) |\varphi'(z) g(z)|}{\phi(|\varphi(z)|) (1-|\varphi(z)|^2)^{1/q + n+1}}, \\ &  \limsup_{|\varphi(z)| \rightarrow 1} \frac{\mu(|z|) | g'(z)|}{\phi(|\varphi(z)|) (1-|\varphi(z)|^2)^{1/q + n}}  \}.
\end{align*}
This completes the proof.
\end{proof}
If we set $n=1$ in the Theorem \ref{th1}, then we have the following corollary.
\begin{cor} \label{th1}
Let  $p,q \in (0,\infty)$, $\phi$ be normal and $\varphi$ be a analytic self-map of $\mathbb{D}$. If the operator $C_{\varphi}^g :H(p,q,\phi) \rightarrow \mathcal{Z}_{\mu}$ is bounded, then
\begin{align*}
  \| C_{\varphi}^g \|_{e} \approx \max  \{ &  \limsup_{|\varphi(z)| \rightarrow 1} \frac{\mu(|z|) |\varphi'(z) g(z)|}{\phi(|\varphi(z)|) (1-|\varphi(z)|^2)^{1/q +2}}, \\ &  \limsup_{|\varphi(z)| \rightarrow 1} \frac{\mu(|z|) | g'(z)|}{\phi(|\varphi(z)|) (1-|\varphi(z)|^2)^{1/q + 1}} \},
\end{align*}
and as a result, $C_{\varphi}^g :H(p,q,\phi) \rightarrow \mathcal{Z}_{\mu}$ is compact if and only if
$$ \limsup_{|\varphi(z)| \rightarrow 1} \frac{\mu(|z|) |\varphi'(z) g(z)|}{\phi(|\varphi(z)|) (1-|\varphi(z)|^2)^{1/q +2}}=0, \ \
\limsup_{|\varphi(z)| \rightarrow 1} \frac{\mu(|z|) | g'(z)|}{\phi(|\varphi(z)|) (1-|\varphi(z)|^2)^{1/q + 1}}=0.  $$
\end{cor}
When $n=1$ and $g(z) = \varphi'(z)$, then $C_{\varphi,\varphi'}^1 = C_{\varphi}$, the well-known composition operator, and we have the next corollary.
\begin{cor} \label{th1}
Let  $p,q \in (0,\infty)$, $\phi$ is normal and $\varphi$ be a analytic self-map of $\mathbb{D}$.  If  the operator $C_{\varphi} :H(p,q,\phi) \rightarrow \mathcal{Z}_{\mu}$ is bounded then
\begin{align*}
  \| C_{\varphi} \|_{e} \approx \max  \{ &  \limsup_{|\varphi(z)| \rightarrow 1} \frac{\mu(|z|) |\varphi'^2(z)|}{\phi(|\varphi(z)|) (1-|\varphi(z)|^2)^{1/q +2}}, \\ &  \limsup_{|\varphi(z)| \rightarrow 1} \frac{\mu(|z|) | \varphi'(z)|}{\phi(|\varphi(z)|) (1-|\varphi(z)|^2)^{1/q + 1}} \},
\end{align*}
and as a result, $C_{\varphi} :H(p,q,\phi) \rightarrow \mathcal{Z}_{\mu}$ is compact if and only if
$$ \limsup_{|\varphi(z)| \rightarrow 1} \frac{\mu(|z|) |\varphi'^2(z) |}{\phi(|\varphi(z)|) (1-|\varphi(z)|^2)^{1/q +2}}=0, \ \
\limsup_{|\varphi(z)| \rightarrow 1} \frac{\mu(|z|) | \varphi'(z)|}{\phi(|\varphi(z)|) (1-|\varphi(z)|^2)^{1/q + 1}}=0.  $$
\end{cor}
With different choices on $\phi$ and $g$ and the other parameters of the operator and
space, we can obtain some other results.
\begin{thm} \label{th2}
Let  $p,q \in (0,\infty)$, $\phi$ is normal and $\varphi$ be a analytic self-map of $\mathbb{D}$.  If  the operator $C_{\varphi,g}^n :H(p,q,\phi) \rightarrow \mathcal{B}_{\mu}$ is bounded then
\begin{align*}
  \| C_{\varphi,g}^n \|_{e} \approx  \limsup_{|\varphi(z)| \rightarrow 1} \frac{\mu(|z|) | g(z)|}{\phi(|\varphi(z)|) (1-|\varphi(z)|^2)^{1/q + n}}.
\end{align*}
\end{thm}
The proof is similar to the proof of Theorem \ref{th1}.
\begin{cor}
  Let  $p,q \in (0,\infty)$, $\phi$ is normal and $\varphi$ be a analytic self-map of $\mathbb{D}$.  Then the operator $C_{\varphi,g}^n :H(p,q,\phi) \rightarrow \mathcal{B}_{\mu}$ is compact if and only if
\begin{align*}
  \limsup_{|\varphi(z)| \rightarrow 1} \frac{\mu(|z|) | g(z)|}{\phi(|\varphi(z)|) (1-|\varphi(z)|^2)^{1/q + n}}=0.
\end{align*}
\end{cor}

\textbf{Declarations}\\
     \textbf{Conflict of interest.} On behalf of all authors, the corresponding author states that there is no conflict of interest.\\
     \textbf{Acknowledgement.} Our manuscript has no associate data.

\end{document}